\documentclass{amsart}
\author{Umar Hayat, \'Alvaro Nolla De Celis and Fawad Ali}
\date{}
\address{Umar Hayat, Department of Mathematics, Quaid-i-Azam University Islamabad, Pakistan}
\email{umar.hayat@qau.edu.pk}
\address{\'Alvaro Nolla de Celis, Universidad Aut\'onoma de Madrid (UAM), Spain}
\email{alvaro.nolla@uam.es}
\address{Fawad Ali, Department of Mathematics, Quaid-i-Azam University Islamabad, Pakistan}
\email{fawadali@math.qau.edu.pk}
\usepackage[T1]{fontenc}
\usepackage{times}%
\usepackage[centertags]{amsmath}
\usepackage{mathtools}
\usepackage{amsmath}
\usepackage{dcolumn}
\usepackage{graphics}
\usepackage{amssymb}
\usepackage{pstricks}
\usepackage{pst-node}
\usepackage{comment}
\usepackage{hyperref}

\newtheorem{Theorem}{\quad Theorem}[section]

\newtheorem{Corollary}[Theorem]{\quad Corollary}
\newtheorem{Proposition}[Theorem]{\quad Proposition}

\newtheorem{Example}[Theorem]{\quad Example}
\numberwithin{equation}{section}

\newcolumntype{d}[1]{D{.}{.}{#1}}

\newcommand{\C}{\mathbb C}
\newcommand{\N}{\mathbb N}
\DeclareMathOperator{\SL}{SL}
\DeclareMathOperator{\Cn}{C}
\DeclareMathOperator{\BD}{BD}
\DeclareMathOperator{\BT}{BT}
\DeclareMathOperator{\BO}{BO}
\DeclareMathOperator{\BI}{BI}
\DeclareMathOperator{\E}{E}
\DeclareMathOperator{\Irr}{Irr}

\usepackage{tikz}
\usepackage{tkz-berge}
\usetikzlibrary{shapes.geometric}

\title{COMMUTING GRAPHS ON COXETER GROUPS, \\DYNKIN DIAGRAMS AND FINITE SUBGROUPS OF $\SL(2,\C)$%\thanks{Footnote in title.}
}

\begin{document}

\begin{abstract}
For a group $H$ and a non empty subset $\Gamma\subseteq H$, the commuting graph $G=\mathcal{C}(H,\Gamma)$ is the graph with $\Gamma$ as the node set and where any $x,y \in \Gamma$ are joined by an edge if $x$ and $y$ commute in $H$. We prove that any simple graph can be obtained as a commuting graph of a Coxeter group, solving the realizability problem in this setup. In particular we can recover every Dynkin diagram of ADE type as a commuting graph. Thanks to the relation between the ADE classification and finite subgroups of $\SL(2,\C)$, we are able to rephrase results from the {\em McKay correspondence} in terms of generators of the corresponding Coxeter groups. We finish the paper studying commuting graphs $\mathcal{C}(H,\Gamma)$ for every finite subgroup $H\subset\SL(2,\C)$ for different subsets $\Gamma\subseteq H$, and investigating metric properties of them when $\Gamma=H$. 
\end{abstract}



\maketitle

\pagestyle{myheadings}

\markboth{U. Hayat, F. Ali, and A. Nolla de Celis}{COMMUTING GRAPHS ON COXETER GROUPS AND FINITE SUBGROUPS $G\subset\SL(2,\C)$}

\section{Introduction}

For a given group $H$ with center $Z(H)$ we can consider the commuting graph $\mathcal{C}(H)$ consisting of the vertex set $H\backslash Z(H)$ and joining two vertices if and only if they commute. This graph has been studied since \cite{BraFo} in several contexts and for different purposes (see \cite{Abdo}, \cite{DolzanOb} among other examples), one of them being to characterize if given a simple graph $G$ whether there exists or not a group $H$ for which $G\cong\mathcal{C}(H)$, known as the {\em realization problem} (see \cite{GiuKuz} and references therein).
 
In this paper we consider commuting graphs $G=\mathcal{C}(H,\Gamma)$ where $\Gamma$ is a nonempty subset of a group $H$, for which $\Gamma$ is the node set and any $x,y \in \Gamma$ are joined by an edge if and only if $x$ and $y$ commute in $H$ (see for example \cite{Bates}, \cite{ASH}). Clearly $\mathcal{C}(H)=\mathcal{C}(H,H)$. In this context, we are able to obtain any simple graph as the commuting graph of a certain Coxeter group $W_M$ taking $\Gamma$ to be the set of generators of $W_M$ (Theorem \ref{Thm1}), therefore solving the realization problem.
 
In addition, Coxeter groups provide us with a rich collection of groups with remarkable connections with other branches of mathematics and sciences. In this paper we are interested in the correspondences between simple laced finite Coxeter groups and finite subgroups of $\SL(2,\C)$ through the simple laced Dynkin diagrams, also known as the {\em ADE classification}. Finite subgroups of $H\subset\SL(2,\C)$ were classified by F. Klein around 1870 into the families of cyclic $\Cn_n$, binary dihedral groups $\BD_{4n}$, and the exceptional cases of binary tetrahedral $\BT_{24}$, binary octahedral $\BO_{48}$ and binary icosahedral $\BI_{120}$. These groups have played a central role not only in group theory but in many other areas such as algebraic geometry, singularity theory, simple Lie algebras and representation theory (see \cite{IN}, \cite{Slodowy} for references and further reading). 

Given a finite $H\subset\SL(2,\C)$ we can consider the orbifold quotient $\C^2/H$, also known as a rational double point or Du Val singularity. J. McKay in \cite{McKay} observed that the resolution graph of the singularity is precisely the (simple laced) Dynkin diagram of the subgroup $H$, initiating the so called {\em McKay Correspondence} (see \cite{Reid}). The realization of (simple laced) Dynkin diagrams as commuting graphs $\mathcal{C}(W_M,\Gamma)$ of Coxeter groups allows us to relate the geometry of the resolution of the singularity $\C^2/H$ with the representation theory of $H$ in terms of the generators $\Gamma$ of $W_M$ (Theorem \ref{Propo}).

We conclude studying commuting graphs $\mathcal{C}(H,\Gamma)$ on finite subgroups $H\subset\SL(2,\C)$ for different subsets $\Gamma\subseteq H$ using the software GAP \cite{GAP4}. For the case $\Gamma=H$, we use this description to compute metric properties of $\mathcal{C}(H)$, in particular the radius, the diameter, their detour analogues and the metric dimension (Theorem \ref{dist1}).

\section{Basic notions}

We denote by $G$ a graph with node set $V(G)$ and edge set $E(G)$. The total number of nodes of the graph $G$, denoted by $|G|$, is called the order of $G$. We write $a\sim b$ if the nodes $a$ and $b$ are adjacent, otherwise $a\nsim b$. 

The join of any two graphs $G_{1}$ and $G_{2}$, denoted by $G_{1} \vee G_{2}$, is the graph with node set $V(G_{1})\cup V(G_{2})$ and edge set $E(G_{1}) \cup E(G_{2}) \cup \{ y \sim z : y \in V(G_{1}), z \in V(G_{2})\}$. The complement of a graph $G$, denoted by $\overline{G}$, is the graph with the same vertex set $V(G)$ where $a\sim b$ in $\overline{G}$ if $a\nsim b$ in $G$. A graph $K_{n}$ with $n$ nodes is called complete graph if any two different nodes of $K_{n}$ are connected by exactly one edge. We denote by $rK_{n}$ the graph consisting of $r$ copies of the complete graph $K_{n}$.

The degree $\deg(a)$ of a node $a$ in a graph $G$ is the number of nodes in $G$ that are incident to $a$ where the loops are counted twice. We may write $\deg_G(a)$ if we want to emphasize the underlying graph. For definitions and further explanations see \cite{4}.

%%%%%%%%%%%%%%%%%%%%%%%%%%%%%%%%%%%%%%
\section{Commuting graphs on Coxeter groups}

Let $M=(m_{ij})_{1\leq i,j\leq n}$ be an $n\times n$ symmetric matrix with $m_{ij}\in\N\cup\{\infty\}$ such that $m_{ii}=1$ and $m_{ij}\geq2$ if $i\neq j$. The {\em Coxeter group of type $M$} is defined as
\[
W_M = \left< s_1,\ldots,s_n  ~|~(s_is_j)^{m_{ij}}=1, ~1\leq i,j\leq n,~m_{ij}<\infty\right>.
\]

See \cite{11} for a standard reference. Observe that condition $m_{ij}=1$ implies that $s_i^2=1$ for all $i$, that is, every generator has order 2. The notation $m_{ij}=\infty$ is reserved for the case when there is no relation of the form $(s_is_j)^m=1$ for any $m$. For example, if $n=1$ then $W_M=\Cn_2$ the cyclic group of order 2 and if $n=2$ then $M=\left(\begin{smallmatrix}1&m\\m&1\end{smallmatrix}\right)$ and $W_M=D_{2m}$ is the dihedral group of order $2m$.

The following result observes that it is possible to realize any simple graph, i.e.\ any unweighted, undirected graph containing no loops or multiple edges, as the commuting graph of an (in general) infinite Coxeter group. 

\begin{Theorem}\label{Thm1} Let $G$ be a simple graph with $n$ vertices. Then $G\cong\mathcal{C}(W_M,\Gamma)$ where $\Gamma =\{ s_1,\ldots, s_n\}$ is the set of generators of a Coxeter group $W_M$.
\end{Theorem}

\begin{proof}
In a Coxeter group $W_M$ any pair of generators $s_i$ and $s_j$ commute if and only if $m_{ij}=2$. Therefore, the commuting graph $\mathcal{C}(W_M,\Gamma)$ is defined as the graph with vertex set $\Gamma$ and where two nodes $s_i$ and $s_j$ are joined if and only if $m_{ij}=2$. This fact allows us to recover $G$ as the commuting graph of the Coxeter group $W_M$ of matrix type $M$ with 
\[
m_{ij} = \left\{ 
\begin{array}{ll} 
1 & \text{, if $i=j$} \\2 & \text{, if $i\sim j$} \\ \geq3 & \text{, if $i\nsim j$.}
 \end{array}\right.
\]  
\end{proof}

\begin{Example} The Petersen graph is the commuting graph $\mathcal{C}(W_M,\Gamma)$ where the Coxeter group $W_M$ is given by the following matrix $M$: 
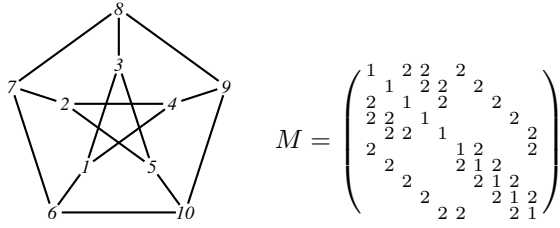
\[
\begin{array}{cc}
\begin{tikzpicture}[rotate=90]
  \GraphInit[vstyle=Empty]
  \SetVertexNoLabel
  \SetUpVertex[MinSize=7pt]
  \grPetersen[RA=1.5,RB=0.75]
  \AssignVertexLabel{a}{\scriptsize{8},\scriptsize{7},\scriptsize{6},\scriptsize{10},\scriptsize{9}}
  \AssignVertexLabel{b}{\scriptsize{3},\scriptsize{2},\scriptsize{1},\scriptsize{5},\scriptsize{4}}
\end{tikzpicture}

&

\begin{tikzpicture}
\node at (0,2) {$
M = \left(\begin{smallmatrix}1&&2&2&&2&&&&\\&1&&2&2&&2&&&\\2&&1&&2&&&2&&\\2&2&&1&&&&&2&\\&2&2&&1&&&&&2\\2&&&&&1&2&&&2\\&2&&&&2&1&2&&\\&&2&&&&2&1&2&\\&&&2&&&&2&1&2\\&&&&2&2&&&2&1\end{smallmatrix}\right)$ 
   };
\end{tikzpicture}

\end{array}
\]

The blank spaces in $M$ correspond to any natural number greater or equal to 3 or $\infty$. Thus there are infinitely many Coxeter groups giving the same commuting graph.
\end{Example}

\subsection{Commuting graphs and Coxeter diagrams} 

Let $W_M$ be a Coxeter group with generator set $\Gamma=\{s_1,\ldots,s_n\}$. The {\em Coxeter diagram} is the graph with vertex set $\Gamma$ where two nodes $i$ and $j$ are joined if and only if $m_{ij}\geq3$. In the case $m_{ij}>3$ the edge $i-j$ is classically labeled by $m_{ij}$, although in what follows we will consider just the {\em Coxeter graph}, denoted by $Cox(W_M)$, which is the Coxeter diagram forgetting the labelling. 

Therefore we can associate to a given Coxeter group $W_M$ two different graphs, $\mathcal{C}(W_M,\Gamma)$ and $Cox(W_M)$. The following result shows the relation between them.

\begin{Proposition} Let $W_M$ be a Coxeter group. Then
\[ Cox(W_M) = \overline{\mathcal{C}(W_M,\Gamma)} \]
\end{Proposition}

\begin{proof} In the Coxeter graph we have that $i\sim j$ if and only if $m_{ij}\geq3$ so the adjacency matrix $A=(a_{ij})$ of $Cox(W_M)$ has every entry equal to 1 except $a_{ij}=0$ when $m_{ij}=1$ or $2$. On the other hand, the adjacency matrix $A'=(a'_{ij})$ of $\mathcal{C}(W_M,\Gamma)$ has entries $a'_{ij}=1$ just when $m_{ij}=2$. Therefore $A+A'$ is the matrix with entries 1 if $i\neq j$ and 0 if $i=j$, so that the corresponding graphs $Cox(W_M)$ and $\mathcal{C}(W_M,\Gamma)$ are complementary.
\end{proof}

In particular $\mathcal{C}(W_M,\Gamma)$ and $Cox(W_M)$ are two simple graphs of $n$ nodes with the same automorphism group and every node corresponds to an involution $s_i\in W_M$. For both graphs the degree at each vertex is related by the following formula.
 
\begin{Corollary} If $W_M$ is a Coxeter group and $\Gamma=\{s_1,\ldots,s_n\}$. Then
\[\deg_{\mathcal{C}(W_M,\Gamma)}(i)+\deg_{Cox(W_M)}(i) = n-1,~~~~~\text{for $i\in{1,\ldots,n}$.} \] 
\end{Corollary}

\begin{proof} Since $\mathcal{C}(W_M,\Gamma)$ and $Cox(W_M)$ are complement graphs we have that
\[ \mathcal{C}(W_M,\Gamma)+Cox(W_M) = \mathcal{C}(W_M,\Gamma)+\overline{\mathcal{C}(W_M,\Gamma)} = K_n \]
where $K_n$ is the complete graph of $n$ vertices. Therefore
\[ \deg_{\mathcal{C}(W_M,\Gamma)}(i)+\deg_{Cox(W_M)}(i) = \deg_{K_n}(i) = n-1.\]
\end{proof}

\begin{Example} Let $W_M$ be the affine Coxeter group $\widetilde{A}_5$, which is defined by the matrix \\$M = \left(\begin{smallmatrix}1&3&2&2&2&3\\3&1&3&2&2&2\\2&3&
1&3&2&2\\2&2&3&1&3&2\\2&2&2&3&1&3\\3&2&2&2&3&1\end{smallmatrix}\right)$. This matrix gives the adjacency matrix $A = \left(\begin{smallmatrix}0&0&1&1&1&0\\0&0&0&1&1&1\\1&0&
0&0&1&1\\1&1&0&0&0&1\\1&1&1&0&0&0\\0&1&1&1&0&0\end{smallmatrix}\right)$ for the graph $\mathcal{C}(\widetilde{A}_5,\Gamma)$. Thus we obtain the complementary graphs:
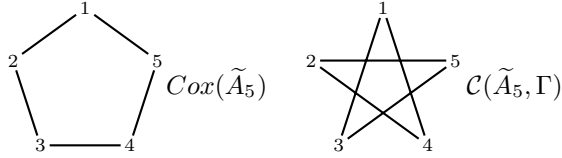
\[
\begin{array}{cc}

\begin{tikzpicture} 
\node[name=s,regular polygon, regular polygon sides=5, minimum size=2cm] at (0,0) {}; 
\node (C1) at (s.corner 1) [inner sep=1.25pt] {\scriptsize $1$};
\node (C2) at (s.corner 2) [inner sep=1.25pt] {\scriptsize $2$};
\node (C3) at (s.corner 3) [inner sep=1.25pt] {\scriptsize $3$};
\node (C4) at (s.corner 4) [inner sep=1.25pt] {\scriptsize $4$};
\node (C5) at (s.corner 5) [inner sep=1.25pt] {\scriptsize $5$};
\draw[-,thick] (C1) -- (C2) -- (C3) -- (C4) -- (C5) -- (C1);
\node at (1.75,0) {$Cox(\widetilde{A}_5)$};
\end{tikzpicture} 

&

\begin{tikzpicture} 
\node[name=s,regular polygon, regular polygon sides=5, minimum size=2cm] at (0,0) {}; 
\node (C1) at (s.corner 1) [inner sep=1.25pt] {\scriptsize $1$};
\node (C2) at (s.corner 2) [inner sep=1.25pt] {\scriptsize $2$};
\node (C3) at (s.corner 3) [inner sep=1.25pt] {\scriptsize $3$};
\node (C4) at (s.corner 4) [inner sep=1.25pt] {\scriptsize $4$};
\node (C5) at (s.corner 5) [inner sep=1.25pt] {\scriptsize $5$};
\draw[-,thick] (C1) -- (C4) -- (C2) -- (C5) -- (C3) -- (C1); 

\node at (1.75,0) {$\mathcal{C}(\widetilde{A}_5,\Gamma)$};
\end{tikzpicture}

\end{array}
\]

\end{Example}

\subsection{Commuting graphs, Dynkin diagrams and the McKay correspondence} 
It is well know that finite and irreducible Coxeter groups are classified by diagrams of types $A_n (n\geq1)$, $B_n (n\geq3)$, $D_n (n\geq 4)$, $E_6$, $E_7$, $E_8$, $F_4$, $G_2$, $H_3$, $H_4$ and $I_2(m)$, which correspond precisely with the semisimple Lie algebras. If we take from this classification only the ones that contains no multiple (or labelled) edges we obtain the so called {\em Dynkin diagrams of ADE type}: $A_n$, $D_n (n\geq 4)$, $E_6$, $E_7$, $E_8$, also named as the {\em ADE classification}. 

Among several correspondences and appearances of these diagrams in other branches of mathematics, we focus in the relation of the {\em ADE classification} with finite subgroups of $\SL(2,\C)$. Namely, cyclic groups $\Cn_n$, binary dihedral groups $\BD_{4n}$, binary tetrahedral $\BT_{24}$, binary octahedral $\BO_{48}$ and binary icosahedral $\BI_{120}$, correspond to diagrams $A_n$, $D_n$, $E_6$, $E_7$ and $E_8$ respectively. 

Firstly, from a group theoretical point of view, Dynkin diagrams of ADE type reveals the relations between irreducible representations of the corresponding group $H\subset\SL(2,\C)$ while tensoring with the natural representation $V$ via the McKay graph (see \ref{Propo} (b) below). Secondly, from an algebraic-geometric point of view, the quotient space $\C^2/H$ has an isolated singularity at the origin for which we can construct its minimal resolution $\pi:Y\to\C^2/H$. The exceptional set $\pi^{-1}(0)=\cup E_i\subset Y$ is a finite union of rational curves, and the Dynkin diagram is the resolution graph of the singularity. In fact, the intersection matrix of the $E_i$ is the negative of the Cartan matrix of the subgroup $H$ (see \cite{McKay}). 

In the following result we rephrase both of the above observations in terms of commuting graphs of Coxeter groups.

\begin{Theorem}\label{Propo} Let $G$ be a Dynkin diagram of ADE type and let $H$ be the corresponding finite subgroup $H\subset\SL(2,\C)$. Then 
\begin{itemize}
\item[(a)] Every Dynkin diagram of ADE type is the commuting graph $\mathcal{C}(W_M,\Gamma)$ of a Coxeter group $W_M$ with $\Gamma=\{s_1,\ldots,s_n\}$.
\item[(b)] Let $H\subset\SL(2,\C)$ be the finite subgroup which corresponds to the graph $G$. Let $\Irr H=\{\rho_0,\ldots,\rho_n\}$ be the set of irreducible representations of $H$, where $\rho_0$ denotes the trivial representation, and let $V$ be the natural representation of $H$. Then 
\[
V\otimes\rho_i = \sum_{j\in J}\rho_j,~~~~~~~\text{where $J=\{j|(s_is_j)^2=1$ with $i\neq j\}$.}
\]  
\item[(c)] The Cartan matrix of $H$ is $C=2I_n-A$ where $I_n$ is the identity $n\times n$ matrix and $A$ is the adjacency matrix of $\mathcal{C}(W_M,\Gamma)$. Moreover, $-C$ is the intersection matrix of the exceptional divisors $E_i$, for $i=1,\ldots,n$, in the minimal resolution of singularities of $C^2/H$.
\end{itemize}
\end{Theorem}

\begin{proof}
Part {\em (a)} is a direct application of Theorem \ref{Thm1}. For part {\em (b)}, we first recall the construction of the McKay graph. For any $\rho_i\in\Irr H$ the representation $V\otimes\rho_i$ decomposes as a sum of irreducible representations as:
\[
V\otimes\rho_i = \sum_{\rho\in\Irr H}\alpha_{ij}\rho_j.
\]
The {\em McKay graph} has vertex set $\Irr H$ and has $\alpha_{ij}$ edges between $\rho_i$ and $\rho_j$. By \cite{McKay} this graph corresponds to the extended (or affine) Dynkin diagram of $H$. Now considering the induced subgraph of $\Irr H\backslash\{\rho_0\}$ and using {\em (a)}, we have the $\alpha_{ij}$ is precisely the $(i,j)$-th entry of the adjacency matrix of $\mathcal{C}(W_M,\Gamma)$. Therefore $\alpha_{ij}=1$ if $s_i$ and $s_j$ commutes in $W_M$, that is $(s_is_j)^2=1$. 

For {\em (c)} again by the McKay correspondence in \cite{McKay} the Dynkin diagram (or equivalently the McKay graph) is the dual graph of $\pi^{-1}(0)$, that is, there is a vertex for every $E_i$ with $i=1,\ldots,n$, and we join $E_i$ with $E_j$ if $E_i\cap E_j\neq\emptyset$. Since $C=2I_n-A$ where $A$ is the adjacency matrix of the Dynkin graph, by {\em (a)} the result follows.
\end{proof}

In Figure \ref{DykinCoxeter} we show every subgroup of $\SL(2,\C)$, their corresponding ADE Dynkin graph and Coxeter matrix $M$. The entries $m_{ij}$ which are not shown consist of integers greater or equal than 3 or $\infty$.

\begin{figure}[h]
\[
\arraycolsep=4pt\def\arraystretch{2.2}
\begin{array}{|c|c|c|c|}
\hline
\text{Type} & H\subset\SL(2,\C) & \mathcal{C}(W_M,\Gamma) & M \\
\hline
\begin{pspicture}(0,1)(1,1.5)
	\rput(0.5,1.35){$A_n$} 
\end{pspicture}
& 
\begin{pspicture}(0,1)(1,1.5)
	\rput(0.5,1.35){$\Cn_n$}
\end{pspicture}	
 & 
\begin{pspicture}(0,1)(1,1.5)
	\psset{arcangle=15,nodesep=2pt}
	\rput(-0.25,0.75){
	\scalebox{0.9}{
	\rput(0,0.75){$\bullet$}\rput(0,0.5){\tiny $1$}
	\rput(0.5,0.75){$\bullet$}\rput(0.5,0.5){\tiny $2$}
	\rput(1.5,0.75){$\bullet$}\rput(1.5,0.5){\tiny $n$}	
	\rput(1.035,0.75){$\cdots$}
	\psline(0,0.75)(0.75,0.75)
	\psline(1.25,0.75)(1.5,0.75)
	}}
\end{pspicture}
&
\begin{pspicture}(0,0)(1,1.25)
	\rput(0.5,0.45){$\left(\begin{smallmatrix}1&2&&&&\\2&1&2&&&\\&2&1&2&&\\&&&\text{\scalebox{0.5}{$\ddots$}}&&\\&&&2&1&2\\&&&&2&1\end{smallmatrix}\right)$}
\end{pspicture} \\

\hline
\begin{pspicture}(0,0)(1.25,1.65)
	\rput(0.65,0.7){$D_{n+2}$} 
\end{pspicture}
& 
\begin{pspicture}(0,0)(1.25,1.65)
	\rput(0.65,0.7){$\BD_{4n}$}
\end{pspicture}	

& 
\begin{pspicture}(0,0)(1.25,1.65)
	\psset{arcangle=15,nodesep=2pt}
	\rput(-0.5,-0.15){
	\scalebox{0.9}{
	\rput(0,0.75){$\bullet$}\rput(0,0.5){\tiny $1$}
	\rput(1,0.75){$\bullet$}\rput(1,0.5){\tiny $n$-$1$}	
	\rput(1.5,0.75){$\bullet$}\rput(1.5,0.5){\tiny $n$}	
	\rput(1.5,1.25){$\bullet$}\rput(1.5,1.5){\tiny $n$+$1$}	
	\rput(2,0.75){$\bullet$}\rput(2,0.5){\tiny $n$+$2$}			
	\rput(0.535,0.75){$\cdots$}
	\psline(0,0.75)(0.25,0.75)
	\psline(0.75,0.75)(2,0.75)
	\psline(1.5,0.75)(1.5,1.25)
	}}
\end{pspicture}
& 
\begin{pspicture}(0,0)(1.25,1.65)
	\rput(0.65,0.65){$\left(\begin{smallmatrix}1&2&&&&&&\\2&1&2&&&&&\\&2&1&2&&&&\\&&&\text{\scalebox{0.5}{$\ddots$}}&&&&\\&&&2&1&2&\\&&&&2&1&2&2\\&&&&&2&1&\\&&&&&2&&1\end{smallmatrix}\right) $}
\end{pspicture} \\

\hline
\begin{pspicture}(0,0)(1.25,1.25)
	\rput(0.65,0.5){$\E_6 $} 
\end{pspicture}
& 
\begin{pspicture}(0,0)(1.25,1.25)
	\rput(0.65,0.5){$\BT_{24}$}
\end{pspicture}	

& 
\begin{pspicture}(0,0)(1.25,1.25)
	\psset{arcangle=15,nodesep=2pt}
	\rput(-0.3,-0.35){
	\scalebox{0.9}{
	\rput(0,0.75){$\bullet$}\rput(0,0.5){\tiny $1$}
	\rput(0.5,0.75){$\bullet$}\rput(0.5,0.5){\tiny $2$}	
	\rput(1,0.75){$\bullet$}\rput(1,0.5){\tiny $3$}	
	\rput(1,1.25){$\bullet$}\rput(1,1.5){\tiny $4$}	
	\rput(1.5,0.75){$\bullet$}\rput(1.5,0.5){\tiny $5$}	
	\rput(2,0.75){$\bullet$}\rput(2,0.5){\tiny $6$}			
	\psline(0,0.75)(2,0.75)
	\psline(1,0.75)(1,1.25)
	}}
\end{pspicture}
& 
\begin{pspicture}(0,0)(1.25,1.25)
	\rput(0.67,0.5){$\left(\begin{smallmatrix}1&2&&&&\\2&1&2&&&\\&2&1&2&2&\\&&2&1&&\\&&2&&1&2\\&&&&2&1\end{smallmatrix}\right)$} 
\end{pspicture} \\

\hline
\begin{pspicture}(0,0)(2.25,1.5)
	\rput(1.15,0.5){$\E_7$} 
\end{pspicture}
& 
\begin{pspicture}(0,0)(2.25,1.5)
	\rput(1.15,0.5){$\BO_{48}$}
\end{pspicture}	

& 
\begin{pspicture}(0,0)(2.25,1.5)
	\scalebox{0.9}{
	\rput(0.05,-0.25){
	\psset{arcangle=15,nodesep=2pt}
	\rput(0,0.75){$\bullet$}\rput(0,0.5){\tiny $1$}
	\rput(0.5,0.75){$\bullet$}\rput(0.5,0.5){\tiny $2$}	
	\rput(1,0.75){$\bullet$}\rput(1,0.5){\tiny $3$}	
	\rput(1,1.25){$\bullet$}\rput(1,1.5){\tiny $4$}	
	\rput(1.5,0.75){$\bullet$}\rput(1.5,0.5){\tiny $5$}	
	\rput(2,0.75){$\bullet$}\rput(2,0.5){\tiny $6$}	
	\rput(2.5,0.75){$\bullet$}\rput(2.5,0.5){\tiny $7$}						
	\psline(0,0.75)(2.5,0.75)
	\psline(1,0.75)(1,1.25)
	}}
\end{pspicture}
&
\begin{pspicture}(0,0)(2.25,1.5)
	\rput(1.2,0.6){$\left(\begin{smallmatrix}1&2&&&&&\\2&1&2&&&&\\&2&1&2&2&&\\&&2&1&&&\\&&2&&1&2\\&&&&2&1&2\\&&&&&2&1\end{smallmatrix}\right) $}
\end{pspicture} \\
 
\hline
\begin{pspicture}(0,0)(2.25,1.6)
	\rput(1.15,0.55){$\E_8$} 
\end{pspicture}
& 
\begin{pspicture}(0,0)(2.25,1.6)
	\rput(1.15,0.55){$\BI_{120}$}
\end{pspicture}	
& 
\begin{pspicture}(0,0.25)(2.75,1.6)
	\psset{arcangle=15,nodesep=2pt}
	\rput(0,-0.){
	\scalebox{0.9}{
	\rput(0,0.75){$\bullet$}\rput(0,0.5){\tiny $1$}
	\rput(0.5,0.75){$\bullet$}\rput(0.5,0.5){\tiny $2$}	
	\rput(1,0.75){$\bullet$}\rput(1,0.5){\tiny $3$}	
	\rput(1,1.25){$\bullet$}\rput(1,1.5){\tiny $4$}	
	\rput(1.5,0.75){$\bullet$}\rput(1.5,0.5){\tiny $5$}	
	\rput(2,0.75){$\bullet$}\rput(2,0.5){\tiny $6$}
	\rput(2.5,0.75){$\bullet$}\rput(2.5,0.5){\tiny $7$}			
	\rput(3,0.75){$\bullet$}\rput(3,0.5){\tiny $8$}			
	\psline(0,0.75)(3,0.75)
	\psline(1,0.75)(1,1.25)
	}}
\end{pspicture}
& 
\begin{pspicture}(0,0)(2.75,1.6)
	\rput(1.4,0.65){$\left(\begin{smallmatrix}1&2&&&&&&\\2&1&2&&&&&\\&2&1&2&2&&&\\&&2&1&&&&\\&&2&&1&2&\\&&&&2&1&2&\\&&&&&2&1&2\\&&&&&&2&1\end{smallmatrix}\right)$}
\end{pspicture} \\
\hline
\end{array}
\]
\caption{Subgroups of $\SL(2,\C)$, ADE Dynkin diagrams $G$ and the corresponding types of Coxeter groups for which $G=\mathcal{C}(W_M,\Gamma)$ with $\Gamma$ the set of generators of $W_M$.}
\label{DykinCoxeter}
\end{figure}

In what follows we concentrate the attention in the finite subgroups $H\subset\SL(2,\C)$ and we investigate commuting graphs $\mathcal{C}(H,\Gamma)$ for relevant subsets $\Gamma\subseteq H$.

%%%%%%%%%%%%%%%%%%%%%%%%%%%%%%%%%
\section{Commuting graphs on finite subgroups of $\SL(2,\C)$}

\subsection{Ciclic subgroups $\Cn_n$}

Since $\Cn_n=Z(\Cn_n)$ this case is trivial. For any subset $\Gamma\subseteq\Cn_n$ we have that $G=\mathcal{C}(\Cn_n,\Gamma)=K_m$ the complete graph with $m=|\Gamma|$.

\subsection{Binary Dihedral subgroups $\BD_{4n}$}

The presentation of binary dihedral groups $\BD_{4n}$ of order $4n$ ($n\geq2$) is given by 
\[
\BD_{4n}=\langle \alpha, \beta : \alpha^{2n}=\beta^{4}=1, \alpha^{n}=\beta^{2}, \alpha\beta=\beta\alpha^{-1} \rangle,
\]

where the center is $Z(\BD_{4n})=\{e, \alpha^{n}\}$. Consider the following subsets of $\BD_{4n}$:
\begin{align*}
\Gamma_{1} & =\{1,\alpha, \alpha^{2}, ... , \alpha^{2n-1}\}, \\
\Gamma_{2} & =\{\beta, \alpha\beta, \alpha^{2}\beta, ... , \alpha^{2n-1}\beta\}, \\
\Gamma_{3} &=\Gamma_{1}\setminus Z(\BD_{4n}).
\end{align*}

\begin{Proposition}\label{propDn}
Let $\BD_{4n}$ be a binary dihedral group and let $G = \mathcal{C}(\BD_{4n},\Gamma)$ be a commuting graph on $\BD_{4n}.$ We have
\[ G\cong \begin{cases}
 K_{2} & \quad{, when ~\Gamma=Z(\BD_{4n}),} \\
  nK_{2} & \quad{, when ~\Gamma=\Gamma_{2},}\\
  K_{2n-2} & \quad{, when ~\Gamma=\Gamma_{3},} \\
  K_{2}\vee (nK_{2}\cup K_{2n-2}) & \quad{, when ~\Gamma=\BD_{4n}.}
\end{cases}\]
\end{Proposition}

\begin{proof}
If $\Gamma=Z(\BD_{4n})=\{e, \alpha^{n}\}$, both elements commute with each other in $\BD_{4n},$ so it is the complete graph $K_{2}$.

If $y=\alpha^{i}\beta$ and $z=\alpha^{n+i}\beta$ with $0\leq i \leq n-1$ are two elements of $\Gamma_{2}$ then $y$ and $z$ commute in $\BD_{4n}$, so $\mathcal{C}(\BD_{4n},\Gamma_2)\cong nK_{2}$. Similarly, any two distinct elements of $\Gamma_{3}$ commute in $\BD_{4n}$, so $\mathcal{C}(\BD_{4n},\Gamma_3)\cong K_{2n-2}$. 

We have that $\BD_{4n}=Z(\BD_{4n})\cup\Gamma_2\cup\Gamma_3$ and notice that elements in $\Gamma_{2}$ and $\Gamma_{3}$ do not commute. Then we can conclude that $\mathcal{C}(\BD_{4n},\BD_{4n})\cong K_{2}\vee (nK_{2}\cup K_{2n-2})$.
\end{proof}

%%%%%%%%%%%%%%%%%%%%%%%%%%%%%%%%%
\subsection{Binary Tetrahedral group $\BT_{24}$}

Let the binary tetrahedral group of order 24 be presented as
\[
\BT_{24} = <r,s,t~|~r^2=s^3=t^3=rst>.
\]
Consider the following subsets of $\BT_{24}$:
\[
\begin{array}{ll}
\mathcal{B}_1 = Z(\BT_{24}) = \{1, r^2\}  & \mathcal{C}_1 = \{s, s^{-1}, s^{-2}, tr\} \\
\mathcal{B}_2 = \{r, r^{-1}\}  & \mathcal{C}_2 = \{rt, r^{-1}t, t^{-1}r^{-1}, t^{-1}r\} \\
\mathcal{B}_3 = \{ts, s^{-1}t^{-1}\} & \mathcal{C}_3 = \{t, t^{-1}, rs, s^{-1}r^{-1}\} \\
\mathcal{B}_4 = \{tsr, s^{-1}t^{-1}r\} & \mathcal{C}_4 = \{sr, t^{-1}s, s^{-1}t, r^{-1}s^{-1}\}. \end{array}
\]

Following GAP \cite{GAP4} computations and using the fact that $\BT_{24}=\bigcup_{i=1}^4\mathcal{B}_i\cup\bigcup_{i=1}^4\mathcal{C}_i$, we obtain the following result.

\begin{Proposition}\label{PropE6}
Let $\BT_{24}$ the binary tetrahedral group and $G=\mathcal{C}(\BT_{24},\Gamma)$ be a commuting graph on $\BT_{24}$. Then we have
\[ 
G = 
\begin{cases}
K_{2} & \quad{, when ~\Gamma=\mathcal{B}_{i} ~~ for~ i=1,\ldots,4}\\
K_{4} & \quad{, when ~\Gamma=\mathcal{C}_{i} ~~ for~ i=1,\ldots,4}\\
K_{2}\vee (3K_2\cup4K_{4}) & \quad{, when ~\Gamma=\BT_{24}.}
\end{cases}
\]
\end{Proposition}

\subsection{Binary Octahedral group $\BO_{48}$}

Let the binary octahedral group be presented as
\[
\BO_{48} = <r,s,t~|~r^2=s^3=t^4=rst>.
\]
Its order is 48 and consider the following subsets of $\BO_{48}$:
\[
\begin{array}{ll}
\mathcal{B}_1 = Z(\BO_{48}) = \{1, r^2\}  & \mathcal{C}_1 = \{s, s^{-1}, s^{-2}, tr\} \\
\mathcal{B}_2 = \{r, r^{-1}\}  & \mathcal{C}_2 = \{rt, r^{-1}t, t^{-1}r^{-1}, t^{-1}r\} \\
\mathcal{B}_3 = \{ts, s^{-1}t^{-1}\} & \mathcal{C}_3 = \{rts, r^{-1}ts, tsr, s^{-1}t^{-1}r\} \\
\mathcal{B}_4 = \{s^{-1}ts^{-1}, st^{-1}s\} & \mathcal{C}_4 = \{srt, t^2s, t^{-2}s, s^{-1}t^2\} \\
\mathcal{B}_5 = \{t^{-1}r^{-1}t,t^{-1}rt\} &  \\
\mathcal{B}_6 = \{r^{-1}tsr,rtsr\} & \mathcal{D}_1 = \{t, t^2, t^{-1}, t^{-2}, rs, s^{-1}r^{-1}\} \\
\mathcal{B}_7 = \{t^{-1}r^{-1}ts,t^{-1}rts\} & \mathcal{D}_2 = \{sr, t^{-1}s, st^{-1}r^{-1}, st^{-1}r, s^{-1}t, r^{-1}s^{-1}\} \\
 & \mathcal{D}_3 = \{t^2r, s^{-1}t^2s, srts, st^{-1}, ts^{-1}, srs\}. 
\end{array}
\]

Following GAP \cite{GAP4} computations and using the fact that $\BO_{48}=\bigcup_{i=1}^7\mathcal{B}_i\cup\bigcup_{i=1}^4\mathcal{C}_i\cup\bigcup_{i=1}^3\mathcal{D}_i$, we obtain the following result.

\begin{Proposition}\label{PropE7}
Let $\BO_{48}$ the binary octahedral group and $\mathcal{C}(\BO_{48},\Gamma)$ be a commuting graph on $\BO_{48}$. Then we have
\[ 
G = 
\begin{cases}
K_{2} & \quad{, when ~\Gamma=\mathcal{B}_{i} ~~ for~ i=1,\ldots,7}\\
K_{4} & \quad{, when ~\Gamma=\mathcal{C}_{i} ~~ for~ i=1,\ldots,4}\\
K_{6} & \quad{, when ~\Gamma=\mathcal{D}_{i} ~~ for~ i=1,\ldots,3}.\\
K_{2}\vee (6K_2\cup4K_{4}\cup3K_{6}) & \quad{, when ~\Gamma=\BO_{48}.}
\end{cases}
\]
\end{Proposition}

\subsection{Binary Icosahedral group $\BI_{120}$}

Let the binary icosahedral group be presented as
\[
\BI_{120} = <r,s,t~|~r^2=s^3=t^5=rst>.
\]
The order of binary icosahedral group is 120 and consider the following subsets of $\BI_{120}$:
\[
\begin{array}{ll}
\mathcal{B}_1 = Z(\BI_{120}) = \{1, r^2\}  &   \\
\mathcal{B}_2 = \{r, r^{-1}\}  & \mathcal{B}_{15} = \{rtsr, r^{-1}tsr \}  \\
\mathcal{B}_3 = \{ts, s^{-1}t^{-1}\} & \mathcal{B}_{16} = \{t^{-1}r^{-1}t, t^{-1}rt \} \\
\mathcal{B}_4 = \{s^{-1}ts^{-1}, st^{-1}s\} &  \\
\mathcal{B}_5 = \{r^{-1}(tsr)^2, (rts)^2r\} & \mathcal{C}_1 = \{s, s^{-1}, s^{-2}, tr\} \\
\mathcal{B}_6 = \{s^{-1}t^2srts, (srt)^2s\} & \mathcal{C}_2 = \{rt, r^{-1}t, t^{-1}r^{-1}, t^{-1}r\} \\
\mathcal{B}_7 = \{t^{-1}r^{-1}tsrts, t^{-1}(rts)^2\} & \mathcal{C}_3 = \{s^{-1}t^{-2}, ts^{-1}t^{-1}, tst^{-1}, t^{2}s\} \\
\mathcal{B}_8 = \{st^{-1}r^{-1}tsr, st^{-1}rtsr\} & \mathcal{C}_4 = \{t^{-1}rt^2s, s^{-1}t^2sr, srtsr, t^{-1}r^{-1}t^2s\} \\
\mathcal{B}_9 = \{st^{-1}r^{-1}t^2s, st^{-1}rt^2s\} & \mathcal{C}_5 = \{ts^{-1}t^{-2}r, st^{-1}r^{-1}ts, st^{-1}rts, t^2st^{-1}r\} \\
\mathcal{B}_{10} = \{t^{-1}r^{-1}tsrt, t^{-1}rtsrt\} & \mathcal{C}_6 = \{t^{-1}r^{-1}t^2, t^{-1}rt^2, ts^{-1}t^{-1}r, t^2sr\}  \\
\mathcal{B}_{11 }= \{tsrts, s^{-1}t^{-1}rts\}  & \mathcal{C}_7 = \{ts^{-1}t^{-2}s, st^{-1}r^{-1}t^2, st^{-1}rt^2, t^2st^{-1}s\} \\
\mathcal{B}_{12} = \{t^2srt, ts^{-1}t^{-1}rt\} & \mathcal{C}_8 = \{st^{-1}r^{-1}t, st^{-1}rt, t^{-1}r^{-1}ts^{-1}, t^{-1}rts^{-1}\} \\
\mathcal{B}_{13} = \{st^{-1}rts^{-1}, st^{-1}r^{-1}ts^{-1}\} & \mathcal{C}_9 = \{t^{-1}r^{-1}ts, s^{-1}t^{-1}rt, t^{-1}rts, tsrt\} \\
\mathcal{B}_{14} = \{t^2st^{-1}, ts^{-1}t^{-2}\} & \mathcal{C}_{10} = \{tst^{-1}r, r^{-1}t^2s, rt^2s, s^{-1}t^{-2}r\}\\
\end{array}
\]

\[
\begin{array}{l}
\mathcal{D}_1 = \{t, t^{-1}, rs, s^{-1}r^{-1}, t^{-2}, t^2, s^{-1}r^{-1}t, t^{-1}rs\} \\
\mathcal{D}_2 = \{sr, r^{-1}s^{-1}, s^{-1}t, t^{-1}s, r^{-1}ts^{-1}, rts^{-1}, st^{-1}r^{-1}, st^{-1}r\} \\
\mathcal{D}_3 = \{srs, st^{-1}, ts^{-1}, t^2r, s^{-1}t^{-2}s, s^{-1}t^2s, srts, tst^{-1}s\} \\
\mathcal{D}_4 = \{tsr, r^{-1}ts, rts, s^{-1}t^{-1}r, r^{-1}tsrts, (rts)^2, s^{-1}t^{-1}rtsr, (tsr)^2\} \\
\mathcal{D}_5 = \{r^{-1}t^2, rt^2, t^{-2}r^{-1}, t^{-2}r, r^{-1}tsrt, rtsrt, t^{-1}r^{-1}tsr, t^{-1}rtsr\} \\ 
\mathcal{D}_6 = \{srt, s^{-1}t^2, t^{-1}r^{-1}s^{-1}, t^{-2}s, s^{-1}t^2srt, (srt)^2, ts^{-1}t^{-1}rts, t^2srts\}. 
\end{array}
\]

Following GAP \cite{GAP4} computations and using the fact that $\BI_{120}=\bigcup_{i=1}^{16}\mathcal{B}_i\cup\bigcup_{i=1}^{10}\mathcal{C}_i\cup\bigcup_{i=1}^{6}\mathcal{D}_i$, we obtain the following result.

\begin{Proposition}\label{PropE8}
Let $\BI_{120}$ the binary icosahedral group and $\mathcal{C}(\BI_{120},\Gamma)$ be a commuting graph on $\BI_{120}$. Then we have
\[ 
G =
\begin{cases}
K_{2} & \quad{, when ~\Gamma=\mathcal{B}_{i} ~~ for~ i=1,\ldots,16}\\
K_{4} & \quad{, when ~\Gamma=\mathcal{C}_{i} ~~ for~ i=1,\ldots,10}\\
K_{8} & \quad{, when ~\Gamma=\mathcal{D}_{i} ~~ for~ i=1,\ldots,6}.\\
K_{2}\vee (15K_2\cup10K_{4}\cup6K_{8}) & \quad{, when ~\Gamma=\BI_{120}.}
\end{cases}
\]
\end{Proposition}

\subsection{Distance properties of commuting graphs on finite subgroups of $\SL(2,\C)$.}
%We refer the reader to \cite{5}, \cite{8} and references therein.

Let $a, b$ be two nodes in a graph $G$. The {\em distance} from $a$ to $b$, denoted by $d(a,b)$, is the length of a shortest path between $a$ to $b$ in $G.$ Also the length of a longest path from $a$ to $b$ in $G$ is denoted by $d_{D}(a,b)$. An $a-b$ path of length $d(a,b)$ is known as a {\em geodesic}, respectively a path of length $d_{D}(a,b)$ is called {\em detour geodesic}.

The largest distance between a node $a$ and any other node of $G$ is called {\em eccentricity}, denoted by $e(a)$. The {\em diameter} $d(G)$ of the graph $G$, is the greatest eccentricity among all the nodes of the graph $G$. Also the {\em radius} $r(G)$ of the graph $G$ is the smallest eccentricity among all the nodes of the graph $G.$ We can define respectively the detour analogous $e_{D}(a)$, $d_{D}(G)$ and $r_{D}(G)$.

For an ordered subset $U=\{u_{1},\ldots, u_{m}\}\subset V(G)$ and a node $t \in G,$ an $m$-vector $r(t|U)=(d(t, u_{1}), \ldots, d(t, u_{m}))$ is said to be the representation of $t$ with respect to $U$. A set $U$ is said be a {\em resolving set} for $G$ if any two different nodes of the graph $G$ have different representation with respect to $U$. A basis of $G$ is a minimum resolving set for the graph $G$, and the {\em metric dimension} $\dim(G)$ is the cardinality of a basis of $G$ (see \cite{CPZ}).

\begin{Theorem}\label{dist1}
Let $H\subset\SL(2,\C)$ be a finite subgroup. Then the radius, diameter, detour radius, detour diameter and metric dimension of the commuting graphs $\mathcal{C}(H)$ for finite $H\subset\SL(2,\C)$ are the following:
\[
\begin{array}{|c|c|c|c|c|c|}
\hline
	& r(G) & d(G) & r_D(G) & d_D(G) & \dim(G) \\
\hline
\mathcal{C}(\Cn_n)	& 1 & 1 & n-1 & n-1 & n-1 \\
\mathcal{C}(\BD_{4n}) & 1 & 2 & 2n+1 & 2n+3 & 3n-2 \\
\mathcal{C}(\BT_{24}) & 1 & 2 & 5 & 13 & 16 \\
\mathcal{C}(\BO_{48}) & 1 & 2 & 7 & 19 & 34 \\
\mathcal{C}(\BI_{120}) & 1 & 2 & 9 & 25 & 88 \\	
\hline
\end{array}
\]
\end{Theorem}

\begin{proof} It is clear that for any group $H$ the graph $G=\mathcal{C}(H)$ contains elements in the center $Z(H)$, so $e(x)=1$ for any $x\in Z(H)$ and $e(x)\leq2$ for $x\notin Z(H)$. Therefore $r(G)=1$ for any $H$, and $d(G)=1$ if $H$ is abelian or 2 if $H$ is not abelian. 

For $r_D$ and $d_D$ we first calculate the eccentricity $e_D$ for each vertex in $G$. If $G=\mathcal{C}(\Cn_n)$ then $e_D(x)=n-1$ for all $x\in\Cn_n$ so $r_D(G)=d_D(G)=n-1$.
 
If $H=\BD_{4n}$ we know by Proposition \ref{propDn} that $G\cong  K_{2}\vee (nK_{2}\cup K_{2n-2})$. Then if $x\in Z(\BD_{4n})$ there is a $x-z$ path of detour length $2n+1$ for every $z\in\Gamma_{2}\cup\Gamma_{3}$ and if $x\in\Gamma_{2}$ (respectively $x\in\Gamma_{2}$) there exists an $x-z$ path of detour length $2n+3$ for every $z\in\Gamma_{3}$ (respectively for every $z\in\Gamma_{3}$). Therefore $r_D(G)=2n+1$ and $d_D(G)=2n+3$. 

If $H=\BT_{24}$ we know by Proposition \ref{PropE6} that $G=\mathcal{C}(\BT_{24})\cong K_{2}\vee (3K_2\cup4K_{4})$ the smallest eccentricity is achieved from an $x-z$ path where $x,z\in Z(\BT_{24})$, and the greatest eccentricity is achieved from an $x-z$ path where $x\in\mathcal{C}_i$ and $x\in\mathcal{C}_j$ with $i\neq j$. Therefore $r_D(G)=5$ and $d_D(G)=13$.

Similarly, for $H=\BO_{48}$ and $H=\BI_{120}$, the smallest eccentricity is achieved from an $x-z$ path where $x,z\in Z(H)$, and the greatest eccentricity is achieved from an $x-z$ path where $x\in\mathcal{D}_i$ and $x\in\mathcal{D}_j$ with $i\neq j$. Thus we obtain the results in the statement.

For the metric dimension $\dim(G)$ we know by \cite{CPZ} that $G\cong K_n$ if and only if $\dim(G)=n-1$. Therefore, if $G=\mathcal{C}(\Cn_n)$ then $\dim(G)=n-1$. 

For the case $H=\BD_{4n}$ we have that $G\cong K_{2}\vee (nK_{2}\cup K_{2n-2})$. Since $\dim(K_2)=1$ and $\dim(K_{2n-2})=2n-3$ we know that there exists a resolving set $U$ of $(n+1)\cdot1+1\cdot(2n-3)=3n-2$ elements, consisting of one element of each $K_2$ component and $2n-3$ elements in $K_{2n-2}$. Clearly, the sets $U_i:=U\setminus\{x_i\}$ for each $x_i\in U$ are not resolving sets, therefore $U$ is a basis and $\dim(G)=3n-2$.

Following the same argument, by Propositions \ref{PropE6}, \ref{PropE7} and \ref{PropE8} we have that
\begin{align*}
\dim(\mathcal{C}(\BT_{24})) &= 4\dim(K_2)+4\dim(K_4) = 16\\
\dim(\mathcal{C}(\BO_{48})) &= 7\dim(K_2)+4\dim(K_4)+3\dim(K_6) = 34\\
\dim(\mathcal{C}(\BT_{120})) &= 16\dim(K_2)+10\dim(K_4)+6\dim(K_8) = 88
\end{align*}
so the result follows.
\end{proof}

%%%%%%%%%%%%%%%%%%%%%%%%%%%%%%%%%%%%%%%%%%%%%%%%%%%%%%%%%%%%%%%%%%%%%%%%%%%%%%%%%%%%%%%%%%%%%%%%%%%%%%%%%%%%%%%%%%%%%%%%%%%%%%%%%%%%%%%%

%%%%%%%%%%%%%%%%%%%

\begin{thebibliography}{12}
\bibitem{Abdo} A. Abdollahi. (2008). Commuting graphs of full matrix rings over finite fields. {\em Linear Algebra and its Applications}, 428:2947--2954.
\bibitem{ASH} F. Ali, M. Salman and S. Huang. (2016). On the commuting graph of dihedral group. {\em Communications in Algebra}, 44:6, 2389--2401.
\bibitem{Bates} C. Bates, D. Bundy, S. Perkins and P. Rowley. (2003). Commuting involution graphs for symmetric groups. {\em J. of Algebra} 266, 133--153.
\bibitem{BraFo} R. Brauer, K.A. Fowler (1965). On groups of even order. {\em Ann. of Math.} 62, 565--583.
\bibitem{CPZ} G. Chartrand, C. Poisson, and P. Zhang. (2000). Resolvability and the upper dimension of graphs, {\em Comput. Math. Appl}. 39, 19--28.
\bibitem{4} G. Chartrand, P. Zhang. (2006). Introduction to graph theory. New York: Tata McGraw-Hill Companies Inc.
%\bibitem{5} G. Chartrand, L. Eroh, M. A. Johnson, O. R. Oellermann. (2000). Resolvability in graphs and the metric dimension of a graph. {\em J. Discrete. Appl. Mathematics}. 105:99--113.
\bibitem{DolzanOb} D. Dol\v zan, P. Oblak. (2011). Commuting graphs of matrices over semirings. {\em J. Linear Algebra and its Applications}, 435:1657--1665.
\bibitem{GAP4} The GAP~Group, \emph{GAP -- Groups, Algorithms, and Programming, Version 4.8.6}; \href{http://www.gap-system.org}{http://www.gap-system.org}.
\bibitem{GiuKuz} M. Giudici, B. Kuzma. (2016). Realizability problem for commuting graphs, {\em Journal of the Australian Mathematical Society} 101, 335--355.
%\bibitem{8}  C. Hernando, M. Mora, I. M. Pelayo, C. Seera, D. R. Wood. (2010). Extremal graph theory for metric dimension and diameter. {\em The Elec. J. of Combinatorics. 17}.
\bibitem{11} J. E. Humphreys (1992). {\em Reflection Groups and Coxeter Groups}. Cambridge Studies in Advanced Mathematics, 29, Cambridge University Press.
\bibitem{IN} Y. Ito, I. Nakamura (1999). Hilbert schemes and simple singularities. {\em In New trends in algebraic geometry (Warwick, 1996)}, Vol.\ 264, London Math. Soc. Lecture Note Ser., 151--233.
\bibitem{McKay} J. McKay (1980). Graphs, singularities, and finite groups. {\em The Santa Cruz Conference on Finite Groups (Univ. California, Santa Cruz, Calif., 1979)}, pp. 183--186, Proc. Sympos. Pure Math., 37, Amer. Math. Soc., Providence, R.I.
%\bibitem{6} M. C. Prajapati. (2011). Distance in graph theory and its application. {\em International J. of Advanced Engineering Technology}. Vol. II, Issue IV, 147--150.
\bibitem{Reid} M. Reid (1999). La correspondance de McKay, S\'eminaire Bourbaki, 52\`eme ann\'ee, 867, Ast\'erisque, 1--21.
\bibitem{Slodowy} P. Slodowy (1980). Simple singularities and simple algebraic groups. Lecture Notes in Mathematics, 815. Springer.
%\bibitem{10} T. Tamizh Chelvam, K. Selvakumar, S. Raja. (2011). Commuting graph on dihedral group. {\em The J. Mathematics and Computer Science}, 2(2):402--406.
\end{thebibliography}
\end{document}